\documentclass{article}
\usepackage{graphicx,amsmath,amsthm,amssymb,subfig,float}
\theoremstyle{definition}
\newtheorem{theorem}{Theorem}[section]

\newtheorem{remark}[theorem]{Remark}
\numberwithin{equation}{section}

\newtheorem{definition}[theorem]{Definition}
\theoremstyle{definition}
\title{Geometric study of Gardner equation}
\author{Mehdi Nadjafikhah\thanks{Corresponding author: Department of Pure Mathematics, School of Mathematics, Iran University of Science and Technology, Narmak, Tehran, 16846-13114, Iran. e-mail:~m\_nadjafikhah@iust.ac.ir} \and Ardavan Mokhtary\thanks{Department of Mathematics, Payame Noor University, Lashkarak, Tehran, 19395-4697, Iran. e-mail:~rdmokhtary@gmail.com}}
\begin{document}
\maketitle
\begin{abstract}
In this paper, we apply the method of approximate transformation groups proposed by Baikov, Gaziziv and Ibragimov \cite{Baikov1,Baikov2}, to compute the first-order approximate symmetry  for the Gardner equations with the small parameters. We compute the optimal system and analyze some invariant solutions of these types of equations. Particularly, general forms of approximately Galilean-invariant solutions have been computed.
\end{abstract}
\section{Introduction} 
The Kortewege-de Vries (KdV) equation is a mathematical model to describe weakly nonlinear long waves. There are many different variations of this PDE \cite{Sch}, But its canonical form is, 
\begin{equation} \label{eq:eq1}
u_{t}-6uu_{x}+u_{xxx}=0.
\end{equation}
The KdV equation  (\ref{eq:eq1}) owes its name to the famous paper of Diederik Kortewedg and Hendrik de Vries, But
much of the significant work on the KdV equation was initiated by the publication of several papers of Gardner et al. (1967-1974). The remarkable discovery of Gardner et al. (1967) that KdV equation was integrable through an inverse scatering transform marked the begining of soliton theory.

In 1968 Miura \cite{Miura1,Miura2} introduce \textit{``Miura transformation"},
\begin{equation} \label{eq:eq2}
u=v^2+v_{x} , 
\end{equation}
 to determinean an infinite number of conservation laws. If we put $v=1/2\epsilon+\epsilon w$ which $\epsilon$ is an arbitrary real parameter, then Miura transformation becomes $ u=1/4\epsilon^2+ w + \epsilon w_{x} + \epsilon^2 w^2$. But, since any arbitrary constant is a trivial solution of KdV equation, it may be removed by a Galilean transformation, so we just consider \textit{``Gardner transfosmation"}, means $u=w+\epsilon w_{x}+\epsilon^2 w^2$, substituting above transformation in KdV equation, observe that w satisfies in \textit{``Gardner equation"},
\begin{equation} \label{eq:eq3}
w_{t}-6(w+\epsilon^2w^2)w_{x}+w_{xxx}=0,
\end{equation}
for all $\epsilon$. Obviously, if $\epsilon=0$, Gardner equation becomes the KdV equation. Next, we observe that the 
for more detailed exposition and references refer to \cite{debnath,Miles}.

These two nonlinear equations are integrable with inverse scattering method. But in this paper, we analyse them with a method which is introduce by Baikov, Gazizov and Ibragimov \cite{Baikov1,Baikov2}.  This method which is known as \textit{``approximate symmetry"} is a combination of Lie group theory and perturbations. There is a second method which is also known as ``approximate symmetry" due to Fushchich
and Shtelen \cite{Fush} and later followed by Euler et al \cite{Euler1,Euler2}. For a comparison of these two methods, we refer the interested reader to the papers \cite{Pak,Wilst1}.
\section{Notations and Definitions}
Before continuing we need to present some definitions and theorems of the book of Ibragimov and Kovalev \cite{Ibra}.

If a function $f(x,\varepsilon)$ satisfies the condition $\lim_{\varepsilon\to 0}f(x,\varepsilon)/\varepsilon^p=0$, it is written $f(x, \varepsilon) = o(\varepsilon^{p})$ and $f$ is said to be of order less than $\varepsilon^{p}$.

If $f(x, \varepsilon) - g(x,\varepsilon) = o(\varepsilon^{p})$, the functions $f$ and $g$ are said to be approximately equal (with an error $o(\varepsilon^{p})$) and written as $f(x,\varepsilon) = g(x,\varepsilon) +o(\varepsilon^p)$, or, briefly $f\approx g$ when there is no ambiguity.
The approximate equality defines an equivalence relation, and we join functions into equivalence classes by letting $f(x,\varepsilon)$ and $g(x,\varepsilon)$ to be members of the same class if and only if $f\approx g$.

Given a function $f(x,\varepsilon)$, let
\begin{equation} \label{eq:eq4}
f_0(x)+\varepsilon f_1(x)+\cdots +\varepsilon^p f_p(x),
\end{equation}
be the approximating polynomial of degree p in $\varepsilon$ obtained via the Taylor series expansion of $f(x, \varepsilon)$ in powers of $\varepsilon$ about $\varepsilon=0$. Then any function $g\approx f$ (in particular, the function $f$ itself) has the form
\begin{eqnarray*}
g(x,\varepsilon) = f_0(x) +\varepsilon f_1(x) \cdots  +\varepsilon^p f_p(x) +o(\varepsilon^p).
\end{eqnarray*}
Consequently the expression (\ref{eq:eq4}) is called a \textit{canonical representative} of the equivalence class of functions containing
$f$· Thus, the equivalence class of functions $g(x,\varepsilon) \approx f(x,\varepsilon)$ is determined by the
ordered set of $p + 1$ functions $f_0(x), f_1(x), \cdots  , f_p(x)$. In the theory of approximate transformation groups, one considers ordered sets of smooth vector-functions depending on $x$'s and a group parameter $a$ as $(f_0(x,a)$, $f_1(x,a)$, $\cdots$, $f_p(x,a))$, with coordinates
$(f_0^{i}(x,a)$, $f_1^i(x,a)$, $\cdots$, $f_p^i(x,a))$, $i= 1, \cdots  ,n$. Let us define the one-parameter family $G$ of approximate transformations
\begin{equation} \label{eq:eq5}
\bar {x}^i\approx f_0^i(x,a)+\varepsilon f_1^i(x,a)+\cdots +\varepsilon^p f_p^i(x,a),\qquad i= 1, \cdots  ,n, 
\end{equation}
of points $x=(x^1,\cdots ,x^n)\in\mathbb{R}^n$ into points $\bar{x}=(\bar{x}^1,\cdots,\bar {x}^n)\in \mathbb{R}^n$ as the class of
invertible transformations
\begin{equation} \label{eq:eq6}
\bar {x} = f(x,a,\varepsilon),
\end{equation}
with vector-functions $f = (f^1,... ,f^n)$ such that $f^i(x,a,\varepsilon)\approx f_0^i(x,a) + \varepsilon f_1^i(x, a) +\cdots +\varepsilon^p f_p^i(x,a)$, $i=1,\cdots,n$. Here $a$ is a real parameter, and the following condition is imposed $f(x,0,\varepsilon)\approx x$.

The set of transformations (\ref{eq:eq5}) is called a \textit{one-parameter approximate transformation group} if $f(f(x,a,\varepsilon),b,\varepsilon) \approx f(x,a+b,\varepsilon)$ for all transformations (\ref{eq:eq6}). Unlike the classical Lie group theory, $f$ does not necessarily denote the same function at each occurrence. It can be replaced by any function $g\approx f$.

Let $G$ be a one-parameter approximate transformation group:
\begin{equation} \label{eq:eq7}
\bar z^i\approx f(z,a,\varepsilon) \equiv f_0^i(z,a) +\varepsilon f_1^i(z,a),\qquad i = 1, \cdots  ,N.
\end{equation}
An approximate equation 
\begin{equation} \label{eq:eq8}
F(z,\varepsilon) \equiv F_0(z)+\varepsilon F_1(z) \approx 0,
\end{equation}
is said to be \textit{approximately invariant with respect to $G$}, or admits $G$ if $F(\bar z,\varepsilon) \approx F(f(z,a,\varepsilon),\varepsilon)=o(\varepsilon)$, whenever $z = (z^1, \cdots  ,z^N)$ satisfies Eq.(\ref{eq:eq8}).  If $z = (x,u,u_{(1)},... ,u_{(k)})$, then (\ref{eq:eq8}) becomes an approximate differential equation of order $k$, and $G$ is an approximate symmetry group of the differential equation.
\begin{theorem} \label{T:thm1}
 Equation (\ref{eq:eq8}) is approximately invariant under the approximate transformation group (\ref{eq:eq7}) with the generator
\begin{equation} \label{eq:eq9}
 X=X_0+\varepsilon X_1\equiv \xi_0^i(z)\,\partial_{z^i}+\varepsilon \xi_1^i(z)\,\partial_{z^i},
\end{equation}
 if and only if $\big[X^{(k)}F(z,\varepsilon)\big]_{F\approx0} = o(\varepsilon)$,  or 
 \begin{equation} \label{eq:eq10}
  \Big[X_0^{(k)}F_0(z)+\varepsilon \Big(X_1^{(k)}F_0(z)+X_0^{(k)}F_1(z)\Big)\Big]_{(\ref{eq:eq8})} = o(\varepsilon).
\end{equation}
\end{theorem}
 In which $k$ is order of equation , and $X^{(k)}$ is $k^{th}$ order prolongation of $X$. The operator (\ref{eq:eq9}) satisfying Eq. (\ref{eq:eq10}) is called an \textit{infinitesimal approximate symmetry} of, or an \textit{approximate operator admitted} by Eq. (\ref{eq:eq8}). Accordingly, Eq. (\ref{eq:eq10}) is termed the \textit{determining equation} for approximate symmetries.
\begin{remark}
\label{rem:rem1}
 The determining equation (\ref{eq:eq10}) can be written as follows:
\begin{equation} \label{eq:eq11}  
X_0^{(k)}F_0(z) = \lambda(z)F_0(z),
\end{equation}
\begin{equation} \label{eq:eq12}
X_1^{(k)} F_0(z) +X_0^{(k)}F_1 (z) = \lambda(z)F_1(z).
\end{equation}
The factor $\lambda(z)$ is determined by Eq. (\ref{eq:eq11}) and then substituted in Eq. (\ref{eq:eq12}). The 
latter equation must hold for all solutions of $F_0(z)=0$.
\end{remark} 
Comparing Eq. (\ref{eq:eq11}) with the determining equation of exact symmetries, we obtain the following statement.
\begin{theorem} \label{T:thm2}
If Eq. (\ref{eq:eq8}) admits an approximate tramformation group with the generator $X = X_0+\varepsilon X_1$, where $X_0\neq 0$, then the operator
\begin{equation} \label{eq:eq13}
X_0=\xi_0^i \,\partial_{z^i},
\end{equation}
is an exact symmetry of the equation
\begin{equation} \label{eq:eq14}
F_0(z)=0.
\end{equation}
\end{theorem} 
\begin{remark} \label{rem:rem2}
It is manifest from Eqs.(\ref{eq:eq11}), (\ref{eq:eq12}) that if $X_0$ is an exact symmetry of Eq.(\ref{eq:eq14}), then $X = \varepsilon X_0$ is an approximate symmetry of Eq. (\ref{eq:eq8}).
\end{remark} 
\begin{definition}
Eqs. (\ref{eq:eq14}) and (\ref{eq:eq8}) are termed an \textit{unperturbed equation} and a
\textit{perturbed equation}, respectively. Under the conditions of Theorem \ref{T:thm2}, the operator
$X_0$ is called a \textit{stable symmetry} of the unperturbed equation (\ref{eq:eq14}). The corresponding
approximate symmetry generator $X = X_0 +\varepsilon X_1$ for the perturbed equation
(\ref{eq:eq13}) is called a \textit{deformation of the infinitesimal symmetry} $X_0$ of Eq. (\ref{eq:eq14}) caused
by the perturbation $\varepsilon F_1(z)$. In particular, if the most general symmetry Lie algebra
of Eq. (\ref{eq:eq14}) is stable, we say that the perturbed equation (\ref{eq:eq8}) \textit{inherits the
symmetries of the unperturbed equation}.
\end{definition}
\section{Approximate Symmetry of Gardner Equation} 
Referring  to Gardner equation (\ref{eq:eq3}), if we put $\varepsilon=\epsilon^2 $ for small real parameter $\epsilon$, it becomes
\begin{equation}
\label{eq:eq15}
w_{t}-6(w+\varepsilon w^2)w_{x}+w_{xxx}=0.
\end{equation}
Now, we can use Remark \ref{rem:rem1} and Theorem \ref{T:thm2}  to provide an infinitesimal method for calculating
approximate symmetries  (\ref{eq:eq9}) for above differential equations with a small parameter.
\subsection{Exact Symmetries} 
Let us consider the approximate group generators in the form
\begin{equation}
\label{eq:eq16}
X =X_0+\varepsilon X_1 = (\alpha_0+\varepsilon \alpha_1)\,\partial_{x}+(\beta_0+\varepsilon \beta_1)\,\partial_{t} +(\eta_0+\varepsilon \eta_1)\,\partial_{w},  
\end{equation}
where $\alpha_i,\beta_i$ and $\eta_i$ for $i = 0, 1$ are unknown functions of $x,t$ and $w$.

Solving the determining equation 
\begin{equation}
\label{eq:eq17}
X_0^{(k)}F_0(z)\Big|_{F_0(z)=0}=0,
\end{equation}
 for the exact symmetries $X_0$ of the unperturbed  Gardner equation, means KdV equation we obtain
\begin{equation}
\label{eq:eq18}
\alpha_0=C_1-6 C_3 t+C_4 x,\qquad \beta_0=C_2+3C_4 t,\qquad \eta_0=C_3-2 C_4 w,
\end{equation}
where $C_1 , \cdots  ,C_4$ are arbitrary constants. Hence,
\begin{equation}
\label{eq:eq19}
X_0=(C_1-6 C_3 t+C_4 x)\,\partial_{x} +(C_2+3C_4 t)\,\partial_{t} +(C_3-2 C_4 w)\,\partial_{w}.
\end{equation}
Therefore, unperturbed Gardner equation, means KdV equation, admits the four-dimensional Lie algebra with the basis
\begin{equation}
\label{eq:eq20}
X_0^1=\,\partial_{x},\quad X_0^2=\,\partial_{t},\quad X_0^3=6t\,\partial_{x}-\,\partial_{w},\quad X_0^4=x\,\partial_{x}+3t\,\partial_{t}-2w\,\partial_{w}.
\end{equation}
\subsection{Approximate Symmetries}
First we need to determine  the auxiliary function $H$ by virtue of Eqs.(\ref{eq:eq11}), (\ref{eq:eq12}) and (\ref{eq:eq8}), i.e., by the equation
\begin{eqnarray*}
 H = \frac{1}{\varepsilon} \big[X_0^{(k)}(F_0(z) +\varepsilon F_1(z))\big|_{F_0(z)+\varepsilon F_1(z)=o}\big].
 \end{eqnarray*}
Substituting the expression (\ref{eq:eq19}) of the generator $X_0$ into above equation
we obtain the auxiliary function
\begin{equation} 
\label{eq:eq21}
H=12 w w_{x} (C_4 w-C_3).
\end{equation} 
Now, calculate the operators $X_1$ by solving the inhomogeneous determining equation
for deformations as $ X_1^{(k)}F_0(z)\big|_{F_0(z)}+H = 0$. This determining equation  for this equation is written as
\begin{eqnarray*}
X_1^{(3)}(w_t-6ww_x+w_{xxx}) +12ww_x(C_4w-C_3) =0.
\end{eqnarray*}
Solving this determining equation yields that $C_4=0$, and hence,
\begin{eqnarray*}
\alpha_1=C_5+C_8x-6C_7t,\qquad\beta_1=C_6 +3C_8 t,\qquad\eta_1=C_7-2(C_3+C_8)w.
\end{eqnarray*}
Then, we obtain the following
approximate symmetries of the Gardner equation:
\begin{eqnarray}
&\mathbf{v}_1=\,\partial_{x},\quad \mathbf{v}_2=\,\partial_{t},\quad
\mathbf{v}_3=6t\,\partial_{x}+(2\varepsilon w-1)\,\partial_{w},\quad\mathbf{v}_4=\varepsilon \mathbf{v}_1, \label{eq:eq22}\\
&\mathbf{v}_5=\varepsilon \mathbf{v}_2,\quad \mathbf{v}_6=\varepsilon(6t\,\partial_{x}-\,\partial_{w})=\varepsilon \mathbf{v}_3,\qquad
\mathbf{v}_7=\varepsilon(x\,\partial_{x}+3t\,\partial_{t}-2w\,\partial_{w}).\nonumber
\end{eqnarray}

Because of $C_4=0$, the scaling operator $X_0^4=x\,\partial_{x}+3t\,\partial_{t}-2w\,\partial_{w}$, is not stable. 
Hence, \textit{the Gardner equation does not inherit the symmetries of the KdV equation}.

In the first-order of precision, We have the following Commutator table,shows that the operators (\ref{eq:eq22}) span an seven-dimensional approximate Lie algebra, and hence generate an seven-parameter approximate transformations group.
\begin{table}[H] 
\centering
\caption{Approximate commutators of Approximate symmetry of Gardner equation} 
\begin{tabular}{l|c*{6}{c}r}
 $[\,,\,]$ & $\mathbf{v}_{1}$ & $\mathbf{v}_{2}$ & $\mathbf{v}_{3}$ & $\mathbf{v}_{4}$ & $\mathbf{v}_{5}$  & $\mathbf{v}_{6}$ & $\mathbf{v}_{7}$ \\
\hline
$\mathbf{v}_{1}$ & 0 & 0 & 0 & 0 & 0 & 0 & $\mathbf{v}_{4}$\\
$\mathbf{v}_{2}$ & 0 & 0 & $6\mathbf{v}_{1}$ & 0 &  0 & $6\mathbf{v}_{4}$ &  $3\mathbf{v}_{5}$  \\
$\mathbf{v}_{3}$ & 0 & $-6\mathbf{v}_{1}$ & 0 & 0 &  $-6\mathbf{v}_{4}$ & 0 &  $-2\mathbf{v}_{6}$  \\
$\mathbf{v}_{4}$ & 0 & 0 & 0 & 0 &  0 & 0 &  0  \\
$\mathbf{v}_{5}$ & 0 & 0 & $6\mathbf{v}_{4}$ & 0 &  0 & 0 &  0  \\
$\mathbf{v}_{6}$ & 0 & $-6\mathbf{v}_{4}$ & 0 & 0 &  0 & 0 &  0  \\
$\mathbf{v}_{7}$ & $-\mathbf{v}_{4}$ & $-3\mathbf{v}_{5}$ & $2\mathbf{v}_{6}$ & 0 &  0 & 0 &  0  \\
\end{tabular}
\end{table}
It is worth noting that the seven-dimensional approximate Lie algebra $\mathfrak{g}$ is solvable and its finite sequence of ideals is as follows:
\begin{eqnarray*}
0\subset\big\langle \mathbf{v}_4\big\rangle \subset\big\langle \mathbf{v}_4,\mathbf{v}_5\big\rangle \subset\big\langle \mathbf{v}_4,\mathbf{v}_5,\mathbf{v}_6\big\rangle \subset\big\langle \mathbf{v}_4,\mathbf{v}_5,\mathbf{v}_6,\mathbf{v}_7\big\rangle \\
\subset\big\langle \mathbf{v}_1,\mathbf{v}_4,\mathbf{v}_5,\mathbf{v}_6,\mathbf{v}_7\big\rangle \subset\big\langle \mathbf{v}_1,\mathbf{v}_2,\mathbf{v}_4,\mathbf{v}_5,\mathbf{v}_6,\mathbf{v}_7\big\rangle \subset\mathfrak{g}.
\end{eqnarray*}
\section{Optimal System for Gardner Equation}
In general, to each $s-$parameter subgroup $H$ of the full symmetry group $G$ of a system of differential equations in $p > s$ independent variables, there will correspond a family of group-invariant solutions. Since there are almost always an infinite number of such subgroups, it is not usually feasible to list all possible group-invariant solutions to the system. We need an effective, systematic means of classifying these solutions, leading to an ``optimal system” of group-invariant solutions from which every other such solution can be derived.
\begin{definition}
Let $G$ be a Lie group with Lie algebra $\mathfrak g$. An optimal system of $s−$parameter subgroups is a list of conjugacy inequivalent $s$−parameter subalgebras with the property that any other subgroup is conjugate to precisely one subgroup in the list. Similarly, a list of $s−$parameter subalgebras forms an optimal system if every $s−$parameter subalgebra of $\mathfrak g$ is equivalent to a unique member of the list under some element of the adjoint representation: $\tilde{\mathfrak{h}} = \textrm{Ad}(g.\mathfrak{h})$, $g \in G$. \cite{Olver1}
\end{definition}
\begin{theorem}
\label{T:thm3}
 Let $H$ and $\tilde{H}$ be connected $s$-dimensional Lie subgroups of the Lie group $G$ with corresponding Lie
subalgebras $\mathfrak{h}$ and $\tilde{\mathfrak{h}}$ of the Lie algebra $g$ of $G$. Then $\tilde{H}=gHg^{-1}$ are conjugate subgroups if and only if $\tilde{\mathfrak{h}} = \textrm{Ad}(g.\mathfrak{h})$ are conjugate subalgebras.(Proposition 3.7 of \cite{Olver1})
\end{theorem}
By theorem (\ref{T:thm3}), the problem of finding an optimal system of subgroups is equivalent to that of finding an optimal
system of subalgebras. To compute the adjoint representation, we use the Lie series
\begin{eqnarray}
\textrm{Ad}(\exp(\mu\mathbf{v}_i))\mathbf{v}_j&=&\mathbf{v}_j-\mu[\mathbf{v}_i,\mathbf{v}_j ] +\frac{\mu^2}{2}[\mathbf{v}_i,[\mathbf{v}_i,\mathbf{v}_j]]+\cdots\nonumber \\
&=&\exp(\textrm{ad}(-\mu\mathbf{v}_i))\mathbf{v}_j,
\end{eqnarray}
where $[\mathbf{v}_i,\mathbf{v}_j]$ is the commutator for the Lie algebra, $\mu$ is a parameter, and $i, j = 1,\ldots,7$. 
In this manner, we construct the table with the $(i,j)$-th entry indicating $F_i^\mu(\mathbf{v}_j):=\textrm{Ad}(\exp(\mu\mathbf{v}_i)\mathbf{v}_j)$. 
\begin{table}\small
\centering
\caption{Adjoint representation  of approximate symmetry of Gardner equation}
 \begin{tabular}{l|c*{6}{c}r}
& $\mathbf{v}_{1}$ & $\mathbf{v}_{2}$ & $\mathbf{v}_{3}$ & $\mathbf{v}_{4}$ & $\mathbf{v}_{5}$  & $\mathbf{v}_{6}$ & $\mathbf{v}_{7}$ \\
\hline
$\mathbf{v}_{1}$ & $\mathbf{v}_{1}$ & $\mathbf{v}_{2}$ & $\mathbf{v}_{3}$ & $\mathbf{v}_{4}$ & $\mathbf{v}_{5}$ & $\mathbf{v}_{6}$ & $\mathbf{v}_{7}-\mu\mathbf{v}_{4}$  \\
$\mathbf{v}_{2}$ & $\mathbf{v}_{1}$ & $\mathbf{v}_{2}$ & $\mathbf{v}_{3}-6\mu \mathbf{v}_{1}$ & $\mathbf{v}_{4}$ & $\mathbf{v}_{5}$ & $\!\!\!\mathbf{v}_{6}-6\mu \mathbf{v}_{4}\!\!\!$ & $\mathbf{v}_{7}-3\mu\mathbf{v}_{5}$ \\
$\mathbf{v}_{3}$ & $\mathbf{v}_{1}$ & $\mathbf{v}_{2}+6\mu \mathbf{v}_{1}$ & $\mathbf{v}_{3}$ & $\mathbf{v}_{4}$ & $\!\!\mathbf{v}_{5}+6\mu \mathbf{v}_{4}\!\!$ & $\mathbf{v}_{6}$ & $\mathbf{v}_{7}+2\mu\mathbf{v}_{6}$ \\
$\mathbf{v}_{4}$ & $\mathbf{v}_{1}$ & $\mathbf{v}_{2}$ & $\mathbf{v}_{3}$ & $\mathbf{v}_{4}$ & $\mathbf{v}_{5}$ & $\mathbf{v}_{6}$ & $\mathbf{v}_{7}$   \\
$\mathbf{v}_{5}$ & $\mathbf{v}_{1}$ & $\mathbf{v}_{2}$ & $\mathbf{v}_{3}-6\mu \mathbf{v}_{4}$ & $\mathbf{v}_{4}$ & $\mathbf{v}_{5}$ & $\mathbf{v}_{6}$ & $\mathbf{v}_{7}$   \\
$\mathbf{v}_{6}$ & $\mathbf{v}_{1}$ & $\mathbf{v}_{2}+6\mu \mathbf{v}_{4}$ & $\mathbf{v}_{3}$ & $\mathbf{v}_{4}$ & $\mathbf{v}_{5}$ & $\mathbf{v}_{6}$ & $\mathbf{v}_{7}$   \\
$\mathbf{v}_{7}$ & $\!\!\mathbf{v}_{1}+\mu \mathbf{v}_{4}\!\!$ & $\mathbf{v}_{2}+3\mu \mathbf{v}_{5}$ & $\mathbf{v}_{3}-2\mu \mathbf{v}_{6}$ & $\mathbf{v}_{4}$ & $\mathbf{v}_{5}$ & $\mathbf{v}_{6}$ & $\mathbf{v}_{7}$  \\
\end{tabular}
\end{table}
\begin{theorem}
\label{T:thm4}
An optimal system of one-dimensional approximate Lie algebras of the Gardner equation is provided by $\mathbf{v}_1$, \ 
$\mathbf{v}_4$, \ $\mathbf{v}_2+\alpha\mathbf{v}_6$, \ $\mathbf{v}_5+\alpha\mathbf{v}_1$, \ 
$\mathbf{v}_3+\alpha \mathbf{v}_2+\beta\mathbf{v}_5$, \ $\mathbf{v}_6+\alpha\mathbf{v}_1+\beta\mathbf{v}_5$, \ 
$\mathbf{v}_7+\alpha\mathbf{v}_1+\beta\mathbf{v}_2+\gamma\mathbf{v}_3$.
\end{theorem}
\begin{proof}
Consider the approximate symmetry algebra $\mathfrak{g}$ of the Gardner equation, whose adjoint representation was determined in the table. 
our task is to simplify as many of the coefficients $a_i$ as possible through judicious applications of adjoint maps to $\mathbf{w}_i$. So that $\mathbf{w}_i$ is equivalent to $\mathbf{w}'_i$ under the adjoint representation. 
 
Given a non-zero vector $\mathbf{w}_1=a_1\mathbf{v}_1 +a_2\mathbf{v}_2\cdots +a_7\mathbf{v}_7$. First suppose that $a_7\neq 0$. Scaling $\mathbf{w}_1$ if necessary, we can assume that $a_7=1$. As for the 7th column of the table, we have:
\begin{eqnarray*}
\mathbf{w}'_1&=&F_1^{a_4-3a_6a_5}\circ F_2^{a_5/3}\circ F_3^{-a_6/2}(\mathbf{w}_1)\\
&=&(a_1-3a_2a_6-2a_3a_5)\mathbf{v}_1+a_2\mathbf{v}_2+a_3\mathbf{v}_3+\mathbf{v}_7.
\end{eqnarray*}
The remaining approximate one-dimensional subalgebras are spanned by vectors of the 
above form with  $a_7=0$. If $a_3\neq 0$, we have $\mathbf{w}_2=a_1\mathbf{v}_1+a_2\mathbf{v}_2+\mathbf{v}_3+
a_4\mathbf{v}_4+a_5\mathbf{v}_5+a_6\mathbf{v}_6$. Next we act on $\mathbf{w}_2$ to cancel the coefficients of $\mathbf{v}_1,\mathbf{v}_4,\mathbf{v}_6$ as follows:
\begin{eqnarray*}
\mathbf{w}'_2=F_2^{a_1/6}\circ F_5^{(a_1a_6+2a_4)/12}\circ F_7^{a_6/2}(\mathbf{w}_2) 
=a_2\mathbf{v}_2+\mathbf{v}_3+\frac{3a_2a_6+2a_5}{2}\mathbf{v}_5.
\end{eqnarray*}
if $a_3,a_7=0$ and $a_2\neq 0$, the non-zero vector $\mathbf{w}_3=a_1\mathbf{v}_1+\mathbf{v}_2+a_4\mathbf{v}_4+a_5\mathbf{v}_5+a_6\mathbf{v}_6$, is equivalent to:
\begin{eqnarray*}
\mathbf{w}'_3=F_3^{-a_1/6}\circ F_6^{(a_1a_5-3a_4)/18}\circ F_7^{-a_5/3}(\mathbf{w}_3)=\mathbf{v}_2+a_6\mathbf{v}_6.
\end{eqnarray*}
if $a_2,a_3,a_7=0$ and $a_6\neq 0$, we scale to make $a_6=1$. Then $\mathbf{w}_4=a_1\mathbf{v}_1+a_4\mathbf{v}_4+a_5\mathbf{v}_5+\mathbf{v}_6$, is 
equivalent to $\mathbf{w}'_4$ under the adjoint representation: $\mathbf{w}'_4=F_2^{a_4/6}(\mathbf{w}_4)=a_1\mathbf{v}_1+a_5\mathbf{v}_5+\mathbf{v}_6$. if $a_2,a_3,a_6,a_7=0$ and $a_5\neq 0$, In the same way as before, the non-zero vector $\mathbf{w}_5=a_1\mathbf{v}_1+a_4\mathbf{v}_4+\mathbf{v}_5$, can be simplified: $\mathbf{w}'_5=F_3^{-a_4/6}(\mathbf{w}_5)=a_1\mathbf{v}_1+\mathbf{v}_5$. if $a_2,a_3,a_5,a_6,a_7=0$ and $a_1\neq 0$, we act on $\mathbf{w}_6=\mathbf{v}_1+a_4\mathbf{v}_4$, by $F_7^{-a_4}$,
 to cancel the coefficient of $\mathbf{w}_4$, leading to $\mathbf{w}'_6=F_6^{-a_4}(\mathbf{w}_6)=\mathbf{v}_1$.
 
The last remaining case occurs when  $a_1,\cdots,a_7=0$ exept $a_4\neq 0$, for which our  earlier simplifications were unnecessary. Because the only remaining vectors are  the multiples of $\mathbf{v}_4$, on which the adjoint representation acts trivially. 
\end{proof}
\section{Approximately differential invariants for the Gardner equation}
In this section we use two different methods to compute an approximately invariant solutions.

In the beginning of this section we compute an approximately invariant solution based on the $X=\mathbf{v}_1+\mathbf{v}_6$.
The approximate invariants for $X$ are determined by the equation: $X(J)=(\,\partial_{x} + \varepsilon(6t\,\partial_{x}-\,\partial_{w}))(J_0+\varepsilon J_1)=o(\varepsilon)$, or equivalently $\partial_{x}(J_0)=0$ and $\partial_{x}(J_1)+(6t\,\partial_{x}-\,\partial_{w})(J_0)=0$. The first equation has two functionally independent solutions $J_0=t $ and  $J_0=w$.

The simplest solutions of the second equation are respectively, $J_1=0$ and $J_1=x$. Therefore  we have two independent invariants $w+\varepsilon x=$ and $\varphi(t)$ respect to $X=\mathbf{v}_1+\mathbf{v}_6$.

Letting $w+\varepsilon x=\varphi(t)$  we obtain the $w=\varphi(t)-\varepsilon x$ for the  approximately invariant solutions. 

 Then, upon substituting  in Gardner equation transform into the equation $\varphi'+6\varepsilon(\varphi-\varepsilon x+\varepsilon(\varphi-\varepsilon x)^2=o(\varepsilon)$. Therefore,  in our approximation we have $\varphi=\varphi'=0$. Thus,  invariant solution to Gardner equation Corresponding to $X$ is $w=-\varepsilon x$. In this manner, we compute functionally approximate invariants respect to the generators of lie algebra and optimal system, as shown in Table 3 below.

Where the unknown functions $F, G, H$ are considered as follows:
\begin{eqnarray*}
F(x,t,w)&=&xw\Big( -1-2\gamma+\frac{3\alpha}{6\gamma t+\alpha}\Big)+x^2\Big(\frac{3\alpha\gamma}{2(6\gamma t+\alpha)^2}-\frac{\gamma^2+\gamma}{6\gamma t+\alpha}\Big),\\
G(x,t,w)&=&-\frac{8\gamma t^3+2\alpha t^2-\beta tx}{\beta},\\
H(x,t,w)&=&-\frac{\gamma(1+2\gamma)t^2}{2\beta}+2(1-\gamma)tw.
\end{eqnarray*}
\begin{table}[H] 
\centering
\caption{Approximate Reduced equations corresponding to infinitesimal symmetries}
 \begin{tabular}{cc*{1}{c}}
 Operator & Approximate Invariants & If \\
\hline
$\mathbf{v}_{1}$ & $t,w$ &
\\ \hline 
$\mathbf{v}_2$ & $x,w$ &
\\ \hline 
$\mathbf{v}_{3}$ & $t,x+6tw-2\varepsilon xw+x^2/6t$ &
\\ \hline 
$\mathbf{v}_4$ & $t,w $  &
\\ \hline 
$\mathbf{v}_{5}$ & $x,w$  &
\\ \hline 
$\mathbf{v}_6$ & $t,x+6tw$ & 
\\ \hline 
$\mathbf{v}_{7}$ & $t/x^3,x^2w$ &
\\ \hline 
$\mathbf{v}_2+\alpha\mathbf{v}_6$ & $x-3\alpha t^2\varepsilon,w+\alpha\varepsilon t$ &
\\ \hline 
$\mathbf{v}_5+\alpha\mathbf{v}_1$ & $\alpha t-\varepsilon$ & $\alpha\neq0$
\\ \hline 
//                                                       & $x,w$ & $\alpha=0$ 
\\ \hline 
$\mathbf{v}_3+\alpha \mathbf{v}_2+\beta\mathbf{v}_5$ & $ \displaystyle 3t^2-\alpha x-\varepsilon\beta x, \atop  \displaystyle t+\alpha w-\varepsilon(x/3-2tw+\beta w)$  
\\ \hline 
$\mathbf{v}_6+\alpha\mathbf{v}_1+\beta\mathbf{v}_5$ & $\alpha t-\varepsilon\beta x,\alpha w+\varepsilon x$ & $\alpha\neq0$
\\ \hline 
//                                                                                      & $3t^2-\beta x, t+\beta w$ & $\alpha=0$
\\ \hline 
$\mathbf{v}_7+\alpha\mathbf{v}_1+\beta\mathbf{v}_2+\gamma\mathbf{v}_3\!\!\!\!\!\!\!\!\!$ & $t/x^3,x^2w$ & $\alpha,\beta,\gamma=0$
\\ \hline 
// & $ \displaystyle t-3\varepsilon tx/(\alpha+6\gamma t),\atop  \displaystyle \gamma(x+6tw)+\alpha w+\varepsilon F(x,t,w)$ & $\beta=0 , \alpha^2+\gamma^2\neq0$
\\ \hline 
// & $ \displaystyle \alpha t-\beta x+3\gamma t^2+\varepsilon G(x,t,w),\atop \displaystyle \beta w+\gamma t +\varepsilon H(x,t,w)$ & $\beta\neq0$
\end{tabular}
\end{table}
Unfortunately, for this method the first-order approximate generator does not necessarily yield a first-order approximate solution. The reason is that the dependent variables are not expanded in a perturbation series \cite{Pak}. Therefore we use another technique to find approximate invariant solutions for the Gardner equation.
\subsection{Approximate Galilean-invariant solution}
 Now, we apply a different technique to find  \textit{Approximate Galilean-invariant solutions} for the Gardner equation. We know that the general form of Galilean-invariant solutions to the unperturbed Gardner equation, means KdV equation, look as $W:=(c-x)/(6t)$. The function $W$ is invariant under the operator $6t\,\partial_{x}-\,\partial_{w}$. 
\begin{figure}
  \centering
  \subfloat[Galilean-invariant solution]{\label{fig:Image1}\includegraphics[width=0.4\textwidth]{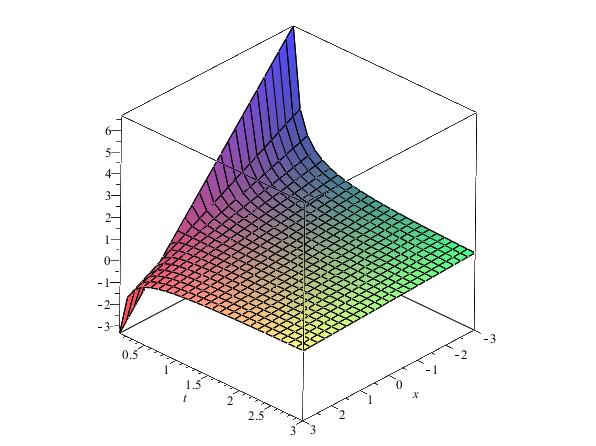}}                
  \subfloat[Approximate Galilean-invariant solution]{\label{fig:Image2}\includegraphics[width=0.35\textwidth]{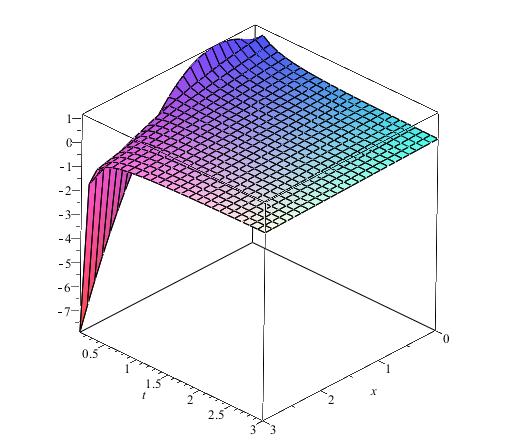}}
  \caption{Comparison of Galilean-invariant and approximate Galilean-invariant solution}
  \label{fig:Galilean}
\end{figure}
We consider the approximate symmetry $X_{\varepsilon} =\mathbf{v}_3+k_1 \mathbf{v}_4+k_2 \mathbf{v}_5+k_3 \mathbf{v}_6+k_4 \mathbf{v}_7$ of the perturbed Gardner equation; That is the approximate symmetry
\begin{eqnarray*}
X_{\varepsilon}&=&6t \,\partial_{x}-\,\partial_{w}+\varepsilon\Big(k_1 \,\partial_{x}+k_2 \,\partial_{t}+
k_3(6t\,\partial_{x}-\,\partial_{w})\\
&&\qquad\qquad\qquad+k_4 (x\,\partial_{x}+3t\,\partial_{t}-2w\,\partial_{w})\Big)=X_0+\varepsilon X_1,
\end{eqnarray*} 
and use it for finding an approximately invariant solution by looking for the invariant perturbation of the: 
\begin{equation} \label{eq:eq23}
w=W+\varepsilon v(x,t).
\end{equation}
The invariant equation for above equation is:
\begin{equation} \label{eq:eq24}
X_\varepsilon \Big(w-W-\varepsilon v(x,t)\Big)\Big|_{(\ref{eq:eq23})}=o(\varepsilon).
\end{equation}
Note that $X_0(w -W)$ vanishes identically. Therefore, Eq. (\ref{eq:eq24}) becomes $X_1(w -W)-X_0(v)=0$. Therefore, we obtain
the following differential equation for $v(t,x)$ as
\begin{eqnarray*}
6t v_t=\frac{k_1+k_4 c}{6t}+k_2\frac{c-x}{6t^2}+\frac{c-x}{3t}.
\end{eqnarray*}
As you can see,  the constant $k_3$ is removed. It is easy to integrate this equation in the ``natural" variables $z=W=(c-x)/(6t)$, \  $y=t$. 
Then it becomes: $v_z+(k_1+k_4 c+6k_2z)/(6y)+2z=0$. The integration yields
\begin{eqnarray*}
v=-\frac{k_1z+k_4 cz+3k_2z^2}{6y}-z^2+F(y).
\end{eqnarray*}
Returning to the variables $t$, $x$, we have 
\begin{eqnarray*}
v=-\frac{(k_1+k_4 c)(c-x)}{36t^2}-k_2\frac{(c-x)^2}{72t^3}-\Big( \frac{c-x}{6t} \Big)^2+F(t).
 \end{eqnarray*}
Inserting this $v$ in (\ref{eq:eq23}) and substituting in the perturbed Gardner equation  we obtain $tF'(t)+F(t)=0$, so $F(t)=C/t$, where $C$ is an arbitrary constant. 

Thus, the approximate symmetry $X_{\varepsilon}$ provides the following the approximately  Galilean-invariant solutions:
 \begin{eqnarray*}
 w=W+\varepsilon\Big(-\frac{(k_1+k_4 c) (c-x)}{36t^2}-k_2\frac{(c-x)^2}{72t^3}-\Big( \frac{c-x}{6t} \Big)^2+\frac{C}{t}  \Big).
 \end{eqnarray*}
Galilean-invariant solution, means $W$, and Approximate Galilean-invariant solution, means $w$, are diplayed below for $-3\leqslant x\leqslant 3$, $0.1\leqslant t \leqslant 3$, $c=C=k_1=k_2=k_4=1$ and $ \varepsilon=0.1$, respectively.

\end{document}